\documentclass[12pt,reqno]{amsart}
\usepackage{amsmath} 

\usepackage{amsthm}
\usepackage{amssymb}
\usepackage{graphics}
\usepackage{latexsym}
\usepackage{comment}

\numberwithin{equation}{section}
\newtheorem{thm}{Theorem}[section]
\newtheorem{prop}[thm]{Proposition}
\newtheorem{lemm}[thm]{Lemma}
\newtheorem{cor}[thm]{Corollary}

\theoremstyle{remark}
\newtheorem{rem}{Remark}[section]
\newtheorem{defn}{Definition}

\newcommand{\BBB}{\mathbb}
\newcommand{\R}{{\BBB R}}
\newcommand{\Z}{{\BBB Z}}

\newcommand{\N}{{\BBB N}}
\newcommand{\C}{{\BBB C}}
\newcommand{\ZZ}{\mathcal{Z}}
\newcommand{\FT}{\mathcal{F}}
\newcommand{\ee}{\mbox{\boldmath $1$}}

\newcommand{\p}{\partial}

\newcommand{\kuuhaku}{\text{}}

\newcommand{\F}{\mathcal{F}}

\title[Well-posedness and scattering for DNLS]{
}

\author[H. Hirayama]{
}
\address[H. Hirayama]{
}
\email[H. Hirayama]{m08035f@math.nagoya-u.ac.jp}


\begin{document}
\begin{center}
{\bf \normalsize WELL-POSEDNESS  AND SCATTERING FOR 
NONLINEAR SCHR\"ODINGER EQUATIONS WITH A DERIVATIVE NONLINEARITY 
AT THE SCALING CRITICAL REGULARITY}

\smallskip \bigskip {\sc Hiroyuki Hirayama}

\smallskip {\footnotesize Graduate School of Mathematics, Nagoya University\\
Chikusa-ku, Nagoya, 464-8602, Japan}
\end{center}
\vspace{-2ex}
\begin{abstract}
In the present paper, we consider the Cauchy problem of nonlinear 
Schr\"odinger equations with a derivative nonlinearity which depends only on $\overline{u}$. 
The well-posedness of the equation at the scaling subcritical regularity was proved by 
A. Gr\"unrock (2000). We prove the well-posedness of the equation and the scattering for the solution 
at the scaling critical regularity  
by using $U^{2}$ space and $V^{2}$ space which are applied to prove 
the well-posedness and the scattering for KP-II equation at the scaling critical regularity by Hadac, Herr and Koch (2009).  
\end{abstract}
\maketitle
\setcounter{page}{001}


\section{Introduction\label{intro}}
We consider the Cauchy problem of the nonlinear Schr\"odinger equations:
\begin{equation}\label{mDNLS}
\begin{cases}
\displaystyle (i\p_{t}+\Delta )u=\p_{k}(\overline{u}^{m}),\hspace{2ex}(t,x)\in (0,\infty )\times \R^{d} \\
u(0,x)=u_{0}(x),\hspace{2ex}x\in \R^{d}
\end{cases}
\end{equation}
where $m\in \N$, $m\geq 2$, $1\leq k\leq d$, $\p_{k}=\partial /\partial x_{k}$ and the unknown function $u$ is $\C$-valued. 
(\ref{mDNLS}) is invariant under the following scaling transformation:
\[
u_{\lambda}(t,x)=\lambda^{-1/(m-1)}u(\lambda^{-2}t,\lambda^{-1}x), 
\]
and the scaling critical regularity is $s_{c}=d/2-1/(m-1)$. 
The aim of this paper is to prove the well-posedness and the scattering for the solution of (\ref{mDNLS}) 
in the scaling critical Sobolev space.

First, we introduce some known results for related problems. 
The nonlinear term in (\ref{mDNLS}) contains a derivative. 
A derivative loss arising from the nonlinearity makes the problem difficult. 
In fact, Mizohata (\cite{Mi85}) proved that a necessary condition for the $L^{2}$ well-posedness of the problem:
\[
\begin{cases}
i\partial_{t}u-\Delta u=b_{1}(x)\cdot \nabla u,\ t\in \R ,\ x\in \R^{d},\\
u(0,x)=u_{0}(x),\ x\in \R^{d}
\end{cases}
\]
is the uniform bound
\[
\sup_{x\in \R^{n},\omega \in S^{n-1},R>0}\left| {\rm Re}\int_{0}^{R}b_{1}(x+r\omega )\cdot \omega dr\right| <\infty.
\]
Furthermore, Christ (\cite{Ch}) proved that the flow map of the Cauchy problem:
\begin{equation}\label{1dqdnls}
\begin{cases}
i\partial_{t}u-\partial_{x}^{2}u=u\partial_{x}u,\ t\in \R,\ x\in \R,\\
u(0,x)=u_{0}(x),\ x\in \R
\end{cases}
\end{equation}
is not continuous on $H^{s}$ for any $s\in \R$. 
While, Ozawa (\cite{Oz98}) proved that the local well-posedness of (\ref{1dqdnls}) in 
the space of all function $\phi \in H^{1}$ satisfying the bounded condition
\[
\sup_{x\in \R}\left|\int_{-\infty}^{x}\phi \right|<\infty .
\]
Furthermore, he proved that if the initial data $\phi$ satisfies some condition, then 
the local solution can be extend globally in time and the solution scatters. 
For the Cauchy problem of the one dimensional derivative Schr\"odinger equation:
\begin{equation}\label{1d_dnls}
\begin{cases}
i\partial_{t}u+\partial_{x}^{2}u=i\lambda \partial_{x}(|u|^{2}u),\ t\in \R,\ x\in \R,\\
u(0,x)=u_{0}(x),\ x\in \R,
\end{cases}
\end{equation}
Takaoka (\cite{Tak99}) proved the local well-posedness in $H^{s}$ for $s\geq 1/2$ by using the gauge transform. 
This result was extended to global well-posedness (\cite{CKSTT01}, \cite{CKSTT02}, \cite{MWX11}, \cite{Tak01}). 
While, ill-posedness of (\ref{1d_dnls}) was obtained for $s<1/2$ (\cite{BL01}, \cite{Tak01}). 
Hao (\cite{Ha07}) considered the Cauchy problem:
\begin{equation}\label{1d_dnls_k}
\begin{cases}
i\partial_{t}u-\partial_{x}^{2}u+i\lambda |u|^{k}\partial_{x}u,\ t\in \R,\ x\in \R,\\
u(0,x)=u_{0}(x),\ x\in \R
\end{cases}
\end{equation}
for $k\geq 5$ and obtained local well-posedness in $H^{1/2}$. 
For more general problem:
\begin{equation}\label{genqdnls}
\begin{cases}
i\partial_{t}u-\Delta u=P(u,\overline{u},\nabla u,\nabla \overline{u}),\ t\in \R,\ x\in \R^{d},\\
u(0,x)=u_{0}(x),\ x\in \R^{d},\\
P\ {\rm is\ a\ polynomial\ which\ has\ no\ constant\ and\ linear\ terms},
\end{cases}
\end{equation}
there are many positive results for the well-posedness 
in the weighted Sobolev space (\cite{Be06}, \cite{Be08}, \cite{Chi95}, \cite{Chi99}, \cite{KPV98}, \cite{St07}). 
Kenig, Ponce and Vega (\cite{KPV98}) also obtained that (\ref{genqdnls}) 
is locally well-posed in $H^{s}$ (without weight)  for large enough $s$ when $P$ has no quadratic terms.

The Benjamin--Ono equation:
\begin{equation}\label{BOeq}
\partial_{t}u+H\partial_{x}^{2}u=u\partial_{x}u,\ (t,x)\in \R \times \R
\end{equation}
is also related to the quadratic derivative nonlinear Schr\"odinger equation. 
It is known that the flow map of (\ref{BOeq}) is not uniformly continuous 
on $H^{s}$ for $s>0$ (\cite{KT05}). 
But the Benjamin--Ono equation has better structure than the equation (\ref{1dqdnls}). 
Actually,  
Tao (\cite{Ta04}) proved that (\ref{BOeq}) is globally well-posed in $H^{1}$ 
by using the gauge transform. 
Furthermore, Ionescu and Kenig (\cite{IK07}) proved that (\ref{BOeq}) is globally well-posed in 
$H_{r}^{s}$ for $s\geq 0$, where $H_{r}^{s}$ is the Banach space of the all real valued function $f\in H^{s}$.  

Next, we introduce some known results for (\ref{mDNLS}). 
Gr\"unrock (\cite{Gr}) proved that (\ref{mDNLS}) is locally well-posed in $L^{2}$ when $d=1$, $m=2$ and 
in $H^{s}$ for $s>s_{c}$ when $d\geq 1$, $m+d\geq 4$. 
Recently, the author (\cite{Hi}) proved that (\ref{mDNLS}) with $d\geq 2$, $m=2$ is globally well-posed for small data in $H^{s_{c}}$ (also in $\dot{H}^{s_{c}}$) 
and the solution scatters. The results are an extension of the results by Gr\"unrock (\cite{Gr}) for $d\geq 2$, $m=2$.  
The main results in this paper are an extension of the results by Gr\"unrock (\cite{Gr}) for $d\geq 1$, $m\geq 3$. 

Now, we give the main results in the present paper. 
For a Banach space $H$ and $r>0$, we define $B_r(H):=\{ f\in H \,|\, \|f\|_H \le r \}$. 
\begin{thm}\label{wellposed_1}Assume $d\geq 1$, $m\geq 3$. \\
{\rm (i)} The equation {\rm (\ref{mDNLS})} is globally well-posed for small data in $\dot{H}^{s_{c}}$. 
More precisely, there exists $r>0$ such that for all initial data $u_{0}\in B_{r}(\dot{H}^{s_{c}})$, there exists a solution
\[
u\in \dot{Z}_{r}^{s_{c}}([0,\infty ))\subset C([0,\infty );\dot{H}^{s_{c}})
\]
of {\rm (\ref{mDNLS})} on $(0, \infty )$. 
Such solution is unique in $\dot{Z}_{r}^{s_{c}}([0,\infty ))$ which is a closed subset of $\dot{Z}^{s_{c}}([0,\infty ))$ {\rm (see Definition~\ref{YZ_space} and (\ref{Zr_norm}))}. 
Moreover, the flow map
\[
S_{+}:B_{r}(\dot{H}^{s_{c}})\ni u_{0}\mapsto u\in \dot{Z}^{s_{c}}([0,\infty ))
\]
is Lipschitz continuous. \\
{\rm (ii)} The statement in {\rm (i)} remains valid if we replace the space $\dot{H}^{s_{c}}$, $\dot{Z}^{s_{c}}([0,\infty ))$ and 
$\dot{Z}_{r}^{s_{c}}([0,\infty ))$ 
by $H^{s}$, $Z^{s}([0,\infty ))$ and $Z_{r}^{s}([0,\infty ))$ for $s\geq s_{c}$.  
\end{thm}
\begin{rem}
Due to the time reversibility of the system (\ref{mDNLS}), the above theorems also hold in corresponding intervals $(-\infty, 0)$. 
We denote the flow map with $t\in (-\infty ,0)$ by $S_{-}$. 
\end{rem}
\begin{cor}\label{sccat}Assume $d\geq 1$, $m\geq 3$. \\
{\rm (i)} Let $r>0$ be as in Theorem~\ref{wellposed_1}. 
For every $u_{0}\in B_{r}(\dot{H}^{s_{c}})$, there exists 
$u_{\pm}\in \dot{H}^{s_{c}}$ 
such that 
\[
S_{\pm}(u_{0})-e^{it\Delta}u_{\pm}
\rightarrow 0
\ {\rm in}\ \dot{H^{s_{c}}}\ {\rm as}\ t\rightarrow \pm \infty. 
\]
{\rm (ii)} The statement in {\rm (i)} remains valid if we replace the space 
$\dot{H}^{s_{c}}$ by $H^{s}$ for $s\geq s_{c}$.  
\end{cor}
The main tools of our results are $U^{p}$ space and $V^{p}$ space which are applied to prove 
the well-posedness and the scattering for KP-II equation at the scaling critical regularity by Hadac, Herr and Koch (\cite{HHK09}, \cite{HHK10}). \\
\kuuhaku \\
\noindent {\bf Notation.} 
We denote the spatial Fourier transform by\ \ $\widehat{\cdot}$\ \ or $\F_{x}$, 
the Fourier transform in time by $\F_{t}$ and the Fourier transform in all variables by\ \ $\widetilde{\cdot}$\ \ or $\F_{tx}$. 
The free evolution $e^{it\Delta}$ on $L^{2}$ is given as a Fourier multiplier
\[
\F_{x}[e^{it\Delta}f](\xi )=e^{-it|\xi |^{2}}\widehat{f}(\xi ). 
\]
We will use $A\lesssim B$ to denote an estimate of the form $A \le CB$ for some constant $C$ and write $A \sim B$ to mean $A \lesssim B$ and $B \lesssim A$. 
We will use the convention that capital letters denote dyadic numbers, e.g. $N=2^{n}$ for $n\in \Z$ and for a dyadic summation we write
$\sum_{N}a_{N}:=\sum_{n\in \Z}a_{2^{n}}$ and $\sum_{N\geq M}a_{N}:=\sum_{n\in \Z, 2^{n}\geq M}a_{2^{n}}$ for brevity. 
Let $\chi \in C^{\infty}_{0}((-2,2))$ be an even, non-negative function such that $\chi (t)=1$ for $|t|\leq 1$. 
We define $\psi (t):=\chi (t)-\chi (2t)$ and $\psi_{N}(t):=\psi (N^{-1}t)$. Then, $\sum_{N}\psi_{N}(t)=1$ whenever $t\neq 0$. 
We define frequency and modulation projections
\[
\widehat{P_{N}u}(\xi ):=\psi_{N}(\xi )\widehat{u}(\xi ),\ 
\widetilde{Q_{M}^{\Delta}u}(\tau ,\xi ):=\psi_{M}(\tau +|\xi|^{2})\widetilde{u}(\tau ,\xi ).
\]
Furthermore, we define $Q_{\geq M}^{\Delta}:=\sum_{N\geq M}Q_{N}^{\Delta}$ and $Q_{<M}^{\Delta}:=Id -Q_{\geq M}^{\Delta}$. 

The rest of this paper is planned as follows.
In Section 2, we will give the definition and properties of the $U^{p}$ space and $V^{p}$ space. 
In Sections 3, we will give the multilinear estimates which are main estimates in this paper. 
In Section 4, we will give the proof of the well-posedness and the scattering (Theorems~\ref{wellposed_1} and Corollary~\ref{sccat}). 
%

\section{$U^{p}$, $V^{p}$ spaces  and their properties \label{func_sp}}
In this section, we define the $U^{p}$ space and the $V^{p}$ space, 
and introduce the properties of these spaces which are proved by Hadac, Herr and Koch (\cite{HHK09}, \cite{HHK10}). 

We define the set of finite partitions $\ZZ$ as
\[
\ZZ :=\left\{ \{t_{k}\}_{k=0}^{K}|K\in \N , -\infty <t_{0}<t_{1}<\cdots <t_{K}\leq \infty \right\}
\]
and if $t_{K}=\infty$, we put $v(t_{K}):=0$ for all functions $v:\R \rightarrow L^{2}$. 
\begin{defn}\label{upsp}
Let $1\leq p <\infty$. For $\{t_{k}\}_{k=0}^{K}\in \ZZ$ and $\{\phi_{k}\}_{k=0}^{K-1}\subset L^{2}$ with 
$\sum_{k=0}^{K-1}||\phi_{k}||_{L^{2}}^{p}=1$ we call the function $a:\R\rightarrow L^{2}$ 
given by
\[
a(t)=\sum_{k=1}^{K}\ee_{[t_{k-1},t_{k})}(t)\phi_{k-1}
\]
a ``$U^{p}${\rm -atom}''. 
Furthermore, we define the atomic space 
\[
U^{p}:=\left\{ \left. u=\sum_{j=1}^{\infty}\lambda_{j}a_{j}
\right| a_{j}:U^{p}{\rm -atom},\ \lambda_{j}\in \C \ {\rm such\ that}\  \sum_{j=1}^{\infty}|\lambda_{j}|<\infty \right\}
\]
with the norm
\[
||u||_{U^{p}}:=\inf \left\{\sum_{j=1}^{\infty}|\lambda_{j}|\left|u=\sum_{j=1}^{\infty}\lambda_{j}a_{j},\ 
a_{j}:U^{p}{\rm -atom},\ \lambda_{j}\in \C\right.\right\}.
\]
\end{defn}
\begin{defn}\label{vpsp}
Let $1\leq p <\infty$. We define the space of the bounded $p$-variation 
\[
V^{p}:=\{ v:\R\rightarrow L^{2}|\ ||v||_{V^{p}}<\infty \}
\]
with the norm
\[
||v||_{V^{p}}:=\sup_{\{t_{k}\}_{k=0}^{K}\in \ZZ}\left(\sum_{k=1}^{K}||v(t_{k})-v(t_{k-1})||_{L^{2}}^{p}\right)^{1/p}.
\]
Likewise, let $V^{p}_{-, rc}$ denote the closed subspace of all right-continuous functions $v\in V^{p}$ with 
$\lim_{t\rightarrow -\infty}v(t)=0$, endowed with the same norm  $||\cdot ||_{V^{p}}$.
\end{defn}
\begin{prop}[\cite{HHK09} Proposition\ 2.2,\ 2.4,\ Corollary\ 2.6]\label{upvpprop}
Let $1\leq p<q<\infty$. \\
{\rm (i)} $U^{p}$, $V^{p}$ and $V^{p}_{-, rc}$ are Banach spaces. \\ 
{\rm (ii)} For Every $v\in V^{p}$, $\lim_{t\rightarrow -\infty}v(t)$ and $\lim_{t\rightarrow \infty}v(t)$ exist in $L^{2}$. \\
{\rm (iii)} The embeddings $U^{p}\hookrightarrow V^{p}_{-,rc}\hookrightarrow U^{q}\hookrightarrow L^{\infty}_{t}(\R ;L^{2}_{x}(\R^{d}))$ are continuous. 
\end{prop}
\begin{thm}[\cite{HHK09} Proposition\ 2,10,\ Remark\ 2.12]\label{duality}
Let $1<p<\infty$ and $1/p+1/p'=1$. 
If $u\in V^{1}_{-,rc}$ be absolutely continuous on every compact intervals, then
\[
||u||_{U^{p}}=\sup_{v\in V^{p'}, ||v||_{V^{p'}}=1}\left|\int_{-\infty}^{\infty}(u'(t),v(t))_{L^{2}(\R^{d})}dt\right|.
\]
\end{thm}
\begin{defn}
Let $1\leq p<\infty$. We define
\[
U^{p}_{\Delta}:=\{ u:\R\rightarrow L^{2}|\ e^{-it\Delta}u\in U^{p}\}
\]
with the norm $||u||_{U^{p}_{\Delta}}:=||e^{-it\Delta}u||_{U^{p}}$, 
\[
V^{p}_{\Delta}:=\{ v:\R\rightarrow L^{2}|\ e^{-it\Delta}v\in V^{p}_{-,rc}\}
\]
with the norm $||v||_{V^{p}_{\Delta}}:=||e^{-it\Delta}v||_{V^{p}}$.
\end{defn}
\begin{rem}
The embeddings $U^{p}_{\Delta}\hookrightarrow V^{p}_{\Delta}\hookrightarrow U^{q}_{\Delta}\hookrightarrow L^{\infty}(\R;L^{2})$ hold for $1\leq p<q<\infty$
by {\rm Proposition~\ref{upvpprop}}. 
\end{rem}
\begin{prop}[\cite{HHK09} Corollary\ 2.18]\label{projest}
Let $1< p<\infty$. We have
\begin{align}
&||Q_{\geq M}^{\Delta}u||_{L_{tx}^{2}}\lesssim M^{-1/2}||u||_{V^{2}_{\Delta}},\label{highMproj}\\
&||Q_{<M}^{\Delta}u||_{V^{p}_{\Delta}}\lesssim ||u||_{V^{p}_{\Delta}},\ \ ||Q_{\geq M}^{\Delta}u||_{V^{p}_{\Delta}}\lesssim ||u||_{V^{p}_{\Delta}},\label{Vproj}
\end{align}
\end{prop}
\begin{prop}[\cite{HHK09} Proposition\ 2.19]\label{multiest}
Let 
\[
T_{0}:L^{2}(\R^{d})\times \cdots \times L^{2}(\R^{d})\rightarrow L^{1}_{loc}(\R^{d})
\]
be a $m$-linear operator. Assume that for some $1\leq p, q< \infty$
\[
||T_{0}(e^{it\Delta}\phi_{1},\cdots ,e^{it\Delta}\phi_{m})||_{L^{p}_{t}(\R :L^{q}_{x}(\R^{d}))}\lesssim \prod_{i=1}^{m}||\phi_{i}||_{L^{2}(\R^{d})}.
\]
Then, there exists $T:U^{p}_{\Delta}\times \cdots \times U^{p}_{\Delta}\rightarrow L^{p}_{t}(\R ;L^{q}_{x}(\R^{d}))$ satisfying
\[
|T(u_{1},\cdots ,u_{m})||_{L^{p}_{t}(\R ;L^{q}_{x}(\R^{d}))}\lesssim \prod_{i=1}^{m}||u_{i}||_{U^{p}_{\Delta}}
\]
such that $T(u_{1},\cdots ,u_{m})(t)(x)=T_{0}(u_{1}(t),\cdots ,u_{m}(t))(x)$ a.e.
\end{prop}
\begin{prop}[Strichartz estimate]\label{Stri_est}
Let $(p,q)$ be an admissible pair of exponents for the Schr\"odinger equation, i.e. $2\leq q\leq 2d/(d-2)$ 
{\rm (}$2\leq q< \infty$ if $d=2$, $2\leq q\leq \infty$ if $d=1${\rm )}, $2/p =d(1/2-1/q)$. Then, we have
\[
||e^{it\Delta}\varphi ||_{L_{t}^{p}L_{x}^{q}}\lesssim ||\varphi ||_{L^{2}_{x}}
\]
for any $\varphi \in L^{2}(\R^{d})$. 
\end{prop}
Proposition~\ref{multiest} and ~\ref{Stri_est} imply the following.
\begin{cor}\label{Up_Stri}
Let $(p,q)$ be an admissible pair of exponents for the Schr\"odinger equation, i.e. $2\leq q\leq 2d/(d-2)$ 
{\rm (}$2\leq q< \infty$ if $d=2$, $2\leq q\leq \infty$ if $d=1${\rm )}, $2/p =d(1/2-1/q)$. Then, we have
\begin{equation}\label{U_Stri}
||u||_{L_{t}^{p}L_{x}^{q}}\lesssim ||u||_{U_{\Delta}^{p}},\ \ u\in U^{p}_{\Delta}.
\end{equation}
\end{cor}
\begin{prop}[\cite{HHK09} Proposition\ 2.20]\label{intpol}
Let $q>1$, $E$ be a Banach space and $T:U^{q}_{\Delta}\rightarrow E$ be a bounded, linear operator 
with $||Tu||_{E}\leq C_{q}||u||_{U^{q}_{\Delta}}$ for all $u\in U^{q}_{\sigma}$.  
In addition, assume that for some $1\leq p<q$ there exists $C_{p}\in (0,C_{q}]$ such that the estimate $||Tu||_{E}\leq C_{p}||u||_{U^{p}_{\Delta}}$ holds true for all $u\in U^{p}_{\Delta}$. Then, $T$ satisfies the estimate
\[
||Tu||_{E}\lesssim C_{p}\left( 1+\ln \frac{C_{q}}{C_{p}}\right) ||u||_{V^{p}_{\Delta}},\ \ u\in V^{p}_{-,rc,\Delta}, 
\]
where implicit constant depends only on $p$ and $q$.
\end{prop}
Next, we define the function spaces which will be used to construct the solution. 
\begin{defn}\label{YZ_space}
Let $s$, $\sigma\in \R$.\\
{\rm (i)} We define $\dot{Z}^{s}:=\{u\in C(\R ; \dot{H}^{s}(\R^{d}))\cap U^{2}_{\Delta}|\ ||u||_{\dot{Z}^{s}}<\infty\}$ with the norm
\[
||u||_{\dot{Z}^{s}}:=\left(\sum_{N}N^{2s}||P_{N}u||^{2}_{U^{2}_{\Delta}}\right)^{1/2}.
\]
{\rm (ii)} We define $Z^{s}:=\{u\in C(\R ; H^{s}(\R^{d}))\cap U^{2}_{\Delta}|\ ||u||_{Z^{s}}<\infty\}$ with the norm
\[
||u||_{Z^{s}}:=||u||_{\dot{Z}^{0}}+||u||_{\dot{Z}^{s}}. 
\]
{\rm (iii)} We define $\dot{Y}^{s}:=\{u\in C(\R ; \dot{H}^{s}(\R^{d}))\cap V^{2}_{-,rc,\Delta}|\ ||u||_{\dot{Y}^{s}}<\infty\}$ with the norm
\[
||u||_{\dot{Y}^{s}}:=\left(\sum_{N}N^{2s}||P_{N}u||^{2}_{V^{2}_{\Delta}}\right)^{1/2}.
\]
{\rm (iv)} We define $Y^{s}:=\{u\in C(\R ; H^{s}(\R^{d}))\cap V^{2}_{-,rc,\Delta}|\ ||u||_{Y^{s}}<\infty\}$ with the norm
\[
||u||_{Y^{s}}:=||u||_{\dot{Y}^{0}}+||u||_{\dot{Y}^{s}}.
\]
\end{defn}
\begin{rem}[\cite{HHK09} Remark\ 2.23]
Let $E$ be a Banach space of continuous functions $f:\R\rightarrow H$, for some Hilbert space $H$. 
We also consider the corresponding restriction space to the interval $I\subset \R$ by
\[
E(I)=\{u\in C(I,H)|\exists v\in E\ s.t.\ v(t)=u(t),\ t\in I\}
\]
endowed with the norm $||u||_{E(I)}=\inf \{||v||_{E}|v(t)=u(t),\ t\in I\}$. 
Obviously, $E(I)$ is also a Banach space. 
\end{rem}
\section{Multilinear estimates \label{Multi_est}}
%
%
In this section, we prove multilinear estimates which will be used to prove the well-posedness.  
\begin{lemm}\label{L2_Multi}
Let $d\geq 1$, $m\geq 2$, $s_{c}=d/2-1/(m-1)$ and $b>1/2$. 
For any dyadic numbers $N_{1}\gg N_{2}\geq \cdots \geq N_{m}$, we have
\begin{equation}\label{L2_Multi_est}
\left|\left|\prod_{j=1}^{m}P_{N_{j}}u_{j}\right|\right|_{L^{2}_{tx}}\lesssim ||P_{N_{1}}u_{1}||_{X^{0,b}}\prod_{j=2}^{m}\left(\frac{N_{j}}{N_{1}}\right)^{1/2(m-1)}N_{j}^{s_{c}}||P_{N_{j}}u_{j}||_{X^{0,b}}, 
\end{equation}
where $||u||_{X^{0,b}}:=||\langle \tau +|\xi|^{2}\rangle^{b}\widetilde{u}||_{L^{2}_{\tau \xi}}$.
\end{lemm}
\begin{proof}
For the case $d=2$ and $m=2$, the estimate (\ref{L2_Multi_est}) is proved by 
Colliander, Delort, Kenig, and Staffilani (\cite{CDKS01} Lemma 1). 
The proof for general case as following is similar to their argument. 

We put $g_{j}(\tau_{j},\xi_{j}):=\langle \tau_{j} +|\xi_{j}|^{2}\rangle^{b}\widetilde{P_{N_{j}}u_{j}}(\tau_{j},\xi_{j})$\ $(j=1,\cdots ,m)$  
and $A_{N}:=\{\xi \in \R^{d} |N/2\leq |\xi |\leq 2N\}$
for a dyadic number $N$. 
By the Plancherel's theorem and the duality argument, it is enough to prove the estimate
\[
\begin{split}
I&:=\left|\int_{\R^{m}}\int_{\prod_{j=1}^{m}A_{N_{j}}}g_{0}\left(\sum_{j=1}^{m}\tau_{j}, \sum_{j=1}^{m}\xi_{j}\right)
\prod_{j=1}^{m}\frac{g_{j}(\tau_{j},\xi_{j})}{\langle \tau_{j}+|\xi_{j}|^{2}\rangle^{b}}
d\xi_{*} d\tau_{*}\right|\\
&\lesssim \left(\prod_{j=2}^{m}\left(\frac{N_{j}}{N_{1}}\right)^{1/2(m-1)}N_{j}^{s_{c}}\right) \prod_{j=0}^{m}||g_{j}||_{L^{2}_{\tau \xi}}
\end{split}
\]
for $g_{0}\in L^{2}_{\tau \xi}$, where $\xi_{*}=(\xi_{1},\cdots ,\xi_{m})$, $\tau_{*}=(\tau_{1},\cdots ,\tau_{m})$.  
We change the variables $\tau_{*}\mapsto \theta_{*}=(\theta_{1}, \cdots , \theta_{m})$ as 
$\theta_{j}=\tau_{j}+|\xi_{j}|^{2}$ $(j=1,\cdots ,m)$
and put 
\[
\begin{split}
G_{0}(\theta_{*}, \xi_{*})&:=g_{0}\left(\sum_{j=1}^{m}(\theta_{j}-|\xi_{j}|^{2}), \sum_{j=1}^{m}\xi_{j}\right),\\
G_{j}(\theta_{j},\xi_{j})&:=g_{j}(\theta_{j}-|\xi_{j}|^{2},\xi_{j})\ (j=1,\cdots ,m).
\end{split}
\]
Then, we have
\[
\begin{split}
I&\leq \int_{\R^{m}}\left(\prod_{j=1}^{m}\frac{1}{\langle \theta_{j}\rangle^{b}}\right)
\left(\int_{\prod_{j=1}^{m}A_{N_{j}}}\left|G_{0}(\theta_{*}, \xi_{*})\prod_{j=1}^{m}G_{j}(\theta_{j}, \xi_{j})\right|d\xi_{*}\right) d\theta_{*}\\
&\lesssim \int_{\R^{m}}\left(\prod_{j=1}^{m}\frac{1}{\langle \theta_{j}\rangle^{b}}\right)
\left(\int_{\prod_{j=1}^{m}A_{N_{j}}}|G_{0}(\theta_{*}, \xi_{*})|^{2}d\xi_{*}\right)^{1/2}\prod_{j=1}^{m}||G_{j}(\theta_{j}, \cdot )||_{L^{2}_{\xi}}d\theta_{*}
\end{split}
\]
by the Cauchy-Schwartz inequality. 
For $1\leq k\leq d$, we put
\[
A_{N_{1}}^{k}:=\{\xi_{1}=(\xi_{1}^{(1)},\cdots, \xi_{1}^{(d)})\in \R^{d}|\ N_{1}/2\leq |\xi_{1}|\leq 2N_{1},\ |\xi_{1}^{(k)}|\geq N_{1}/(2\sqrt{d})\}
\]
and 
\[
J_{k}(\theta_{*}):=\left(\int_{A_{N_{1}}^{k}\times \prod_{j=2}^{m}A_{N_{j}}}|G_{0}(\theta_{*}, \xi_{*})|^{2}d\xi_{*}\right). 
\]
We consider only the estimate for $J_{1}$. The estimates for other $J_{k}$ are obtained by the same way. 

Assume $d\geq 2$. 
By changing the variables $(\xi_{1}, \xi_{2})=(\xi_{1}^{(1)},\cdots, \xi_{1}^{(d)}, \xi_{2}^{(1)},\cdots, \xi_{2}^{(d)}) \mapsto (\mu, \nu, \eta )$ as
\begin{equation}\label{ch_var}
\begin{cases}
\mu =\sum_{j=1}^{m}(\theta_{j}-|\xi_{j}|^{2})\in \R , \\
\nu =\sum_{j=1}^{m}\xi_{j}\in \R^{d},\\
\eta=(\xi_{2}^{(2)}\cdots, \xi_{2}^{(d)})\in \R^{d-1}, 
\end{cases}
\end{equation}
we have 
\[
d\mu d\nu d\eta=2|\xi_{1}^{(1)}-\xi_{2}^{(1)}|d\xi_{1}d\xi_{2}
\]
and
\[
G_{0}(\theta_{*}, \xi_{*})=g_{0}(\mu, \nu ).
\]
We note that $|\xi_{1}^{(1)}-\xi_{2}^{(1)}|\sim N_{1}$ for any $(\xi_{1}, \xi_{2})\in A_{N_{1}}^{1}\times A_{N_{2}}$ with $N_{1}\gg N_{2}$. 
Furthermore, $\xi_{2}\in A_{N_{2}}$ implies that $\eta \in [-2N_{2}, 2N_{2}]^{d-1}$.  Therefore,  we obtain
\[
\begin{split}
J_{1}(\theta_{*})&\lesssim 
\int_{\prod_{j=3}^{m}A_{N_{j}}}\left(\int_{[-2N_{2}, 2N_{2}]^{d-1}} \int_{\R^{d}} \int_{\R} |g_{0}(\mu,\nu)|^{2} \frac{1}{N_{1}}d\mu d\nu d\eta \right)
d\xi_{3}\cdots d\xi_{m}\\
&\sim\frac{N_{2}^{d-1}}{N_{1}}\left(\prod_{j=3}^{m}N_{j}^{d}\right) ||g_{0}||_{L^{2}_{\tau \xi}}^{2}
\leq \left(\prod_{j=2}^{m}\left(\frac{N_{j}}{N_{1}}\right)^{1/(m-1)}N_{j}^{d-2/(m-1)}\right) ||g_{0}||_{L^{2}_{\tau \xi}}^{2}
\end{split}
\]
since $N_{2}\geq N_{j}$ for $3\leq j\leq m$. As a result, we have
\[
\begin{split}
I&\lesssim\int_{\R^{m}}\left(\prod_{j=1}^{m}\frac{1}{\langle \theta_{j}\rangle^{b}}\right)
\left(\sum_{k=1}^{d}J_{k}(\theta_{*})\right)^{1/2}\prod_{j=1}^{m}||G_{j}(\theta_{j}, \cdot )||_{L^{2}_{\xi}}d\theta_{*}\\
&\lesssim \left(\prod_{j=2}^{m}\left(\frac{N_{j}}{N_{1}}\right)^{1/2(m-1)}N_{j}^{s_{c}}\right) \prod_{j=0}^{m}||g_{j}||_{L^{2}_{\tau \xi}}
\end{split}
\]
by the Cauchy-Schwartz inequality and changing the variables $\theta_{*}\mapsto \tau_{*}$ as $\theta_{j}=\tau_{j}+|\xi_{j}|^{2}$ $(j=1,\cdots ,m)$. 

For $d=1$, we obtain the same result by changing the variables $(\xi_{1}, \xi_{2}) \mapsto (\mu, \nu)$ as
$\mu =\sum_{j=1}^{m}(\theta_{j}-|\xi_{j}|^{2})$, $\nu =\sum_{j=1}^{m}\xi_{j}$ instead of (\ref{ch_var}). 
\end{proof}
\begin{cor}\label{UV_Multi}
Let $m\geq 2$, $m+d\geq 4$ and $s_{c}=d/2-1/(m-1)$. 
For any dyadic numbers $N_{1}\gg N_{2}\geq \cdots \geq N_{m}$ and $0<\delta <1/2(m-1)$, we have
\begin{align}
&\left|\left|\prod_{j=1}^{m}P_{N_{j}}u_{j}\right|\right|_{L^{2}_{tx}}\lesssim ||P_{N_{1}}u_{1}||_{U^{2}_{\Delta}}\prod_{j=2}^{m}\left(\frac{N_{j}}{N_{1}}\right)^{1/2(m-1)}N_{j}^{s_{c}}||P_{N_{j}}u_{j}||_{U^{2}_{\Delta}},\label{U_Multi}\\
&\left|\left|\prod_{j=1}^{m}P_{N_{j}}u_{j}\right|\right|_{L^{2}_{tx}}\lesssim ||P_{N_{1}}u_{1}||_{V^{2}_{\Delta}}\prod_{j=2}^{m}\left(\frac{N_{j}}{N_{1}}\right)^{\delta}N_{j}^{s_{c}}||P_{N_{j}}u_{j}||_{V^{2}_{\Delta}}.\label{V_Multi}
\end{align}
\end{cor}
\begin{proof}
To obtain (\ref{U_Multi}), we use the argument of the proof of Corollary 2.21\ (27) in \cite{HHK09}. 
Let $\phi_{1}$, $\cdots$, $\phi_{m}\in L^{2}(\R^{d})$ and define $\phi_{j}^{\lambda}(x):=\phi_{j}(\lambda x)$ $(j=1,\cdots ,m)$
for $\lambda \in \R$. By using the rescaling $(t,x)\mapsto (\lambda^{2}t, \lambda x)$, we have
\[
\begin{split}
&\left|\left|\prod_{j=1}^{m}P_{N_{j}}(e^{it\Delta}\phi_{j})\right|\right|_{L^{2}([-T,T]\times \R^{d})}
=\lambda^{d/2+1}\left|\left|\prod_{j=1}^{m}P_{\lambda N_{j}}(e^{it\Delta}\phi_{j}^{\lambda})\right|\right|_{L^{2}([-\lambda^{-2}T,\lambda^{-2}T]\times \R^{d})}. 
\end{split}
\]
Therefore by putting $\lambda =\sqrt{T}$ and (\ref{L2_Multi_est}), we have
\[
\begin{split}
&\left|\left|\prod_{j=1}^{m}P_{N_{j}}(e^{it\Delta}\phi_{j})\right|\right|_{L^{2}([-T,T]\times \R^{d})}\\
&\lesssim \sqrt{T}^{md/2}||P_{\sqrt{T}N_{1}}\phi_{1}^{\sqrt{T}}||_{L^{2}_{x}}\prod_{j=2}^{m}\left(\frac{N_{j}}{N_{1}}\right)^{1/2(m-1)}N_{j}^{s_{c}}||P_{\sqrt{T}N_{j}}\phi_{j}^{\sqrt{T}}||_{L^{2}_{x}}\\
&=||P_{N_{1}}\phi_{1}||_{L^{2}_{x}}\prod_{j=2}^{m}\left(\frac{N_{j}}{N_{1}}\right)^{1/2(m-1)}N_{j}^{s_{c}}||P_{N_{j}}\phi_{j}||_{L^{2}_{x}}.
\end{split}
\]
Let $T\rightarrow \infty$, then we obtain 
\[
\left|\left|\prod_{j=1}^{m}P_{N_{j}}(e^{it\Delta}\phi_{j})\right|\right|_{L^{2}_{tx}}
\lesssim ||P_{N_{1}}\phi_{1}||_{L^{2}_{x}}\prod_{j=2}^{m}\left(\frac{N_{j}}{N_{1}}\right)^{1/2(m-1)}N_{j}^{s_{c}}||P_{N_{j}}\phi_{j}||_{L^{2}_{x}}
\]
and (\ref{U_Multi}) follows from proposition~\ref{multiest}. 

To obtain (\ref{V_Multi}), we first prove the $U^{2m}$ estimate. By the Cauchy-Schwartz inequality, the Sobolev embedding $\dot{W}^{s_{c}, 2md/(md-2)}(\R^{d})\hookrightarrow L^{m(m-1)d}(\R^{d})$ (which holds when $m\geq 2$, $m+d\geq 4$) and (\ref{U_Stri}), we have
\begin{equation}\label{Up_est}
\begin{split}
\left|\left|\prod_{j=1}^{m}P_{N_{j}}u_{j}\right|\right|_{L^{2}_{tx}}
&\lesssim ||P_{N_{1}}u_{1}||_{L^{2m}_{t}L^{2md/(md-2)}_{x}}\prod_{j=2}^{m}N_{j}^{s_{c}}||P_{N_{j}}u_{j}||_{L^{2m}_{t}L^{2md/(md-2)}_{x}}\\
&\lesssim ||P_{N_{1}}u_{1}||_{U^{2m}_{\Delta}}\prod_{j=2}^{m}N_{j}^{s_{c}}||P_{N_{j}}u_{j}||_{U^{2m}_{\Delta}}
\end{split}
\end{equation}
for any dyadic numbers $N_{1}$, $\cdots$, $N_{m}\in 2^{\Z}$. 
We use the interpolation between (\ref{U_Multi}) and (\ref{Up_est}) via 
Proposition~\ref{intpol}. Then, we get (\ref{V_Multi}) 
by the same argument of the proof of Corollary 2.21\ (28) in \cite{HHK09}. 
\end{proof}
\begin{lemm}\label{modul_est}
We assume that $(\tau_{0},\xi_{0})$, $(\tau_{1}, \xi_{1})$, $\cdots$, $(\tau_{m}, \xi_{m})\in \R\times \R^{d}$ satisfy 
$\sum_{j=0}^{d}\tau_{j}=0$ and $\sum_{j=0}^{d}\xi_{j}=0$. Then, we have 
\begin{equation}\label{modulation_est}
\max_{0\leq j\leq m}|\tau_{j}+|\xi_{j}|^{2}|
\geq \frac{1}{m+1}\max_{0\leq j\leq m}|\xi_{j}|^{2}. 
\end{equation}
\end{lemm}
\begin{proof}By the triangle inequality, we obtain (\ref{modulation_est}).  
\end{proof}
The following propositions will be used to prove the key estimate for the 
well-posedness in the next section.   
\begin{prop}\label{HL_est}
Let $d\geq 1$, $m\geq 3$, $s_{c}=d/2-1/(m-1)$ and $0<T\leq \infty$. 
For a dyadic number $N_{1}\in 2^{\Z}$, we define the set $S(N_{1})$ as
\[
S(N_{1}):=\{ (N_{2},\cdots ,N_{m})\in (2^{\Z})^{m-1}|N_{1}\gg N_{2}\geq \cdots \geq N_{m}\}. 
\]
If $N_{0}\sim N_{1}$, then we have
\begin{equation}\label{hl}
\begin{split}
&\left|\sum_{S(N_{1})}\int_{0}^{T}\int_{\R^{d}}\left(N_{0}\prod_{j=0}^{m}P_{N_{j}}u_{j}\right)dxdt\right|\\
&\lesssim 
||P_{N_{0}}u_{0}||_{V^{2}_{\Delta}}||P_{N_{1}}u_{1}||_{V^{2}_{\Delta}}\prod_{j=2}^{m}||u_{j}||_{\dot{Y}^{s_{c}}}. 
\end{split}
\end{equation}
\end{prop}
\begin{proof} 
We define $u_{j,N_{j},T}:=\ee_{[0,T)}P_{N_{j}}u_{j}$\ $(j=1,\cdots ,m)$ and put 
$M:=N_{0}^{2}/4(m+1)$. We decompose $Id=Q^{\Delta}_{<M}+Q^{\Delta}_{\geq M}$. 
We divide the integrals on the left-hand side of (\ref{hl}) into $2^{m+1}$ piece of the form 
\begin{equation}\label{piece_form_hl}
\int_{\R}\int_{\R^{d}}\left(N_{0}\prod_{j=0}^{m}Q_{j}^{\Delta}u_{j,N_{j},T}\right) dxdt
\end{equation}
with $Q_{j}^{\Delta}\in \{Q_{\geq M}^{\Delta}, Q_{<M}^{\Delta}\}$\ $(j=0,\cdots ,m)$. 
By the Plancherel's theorem, we have
\[
(\ref{piece_form_hl})
= c\int_{\sum_{j=0}^{m}\tau_{j}=0}\int_{\sum_{j=0}^{m}\xi_{j}=0}N_{0}\prod_{j=0}^{m}\FT[Q_{j}^{\Delta}u_{j,N_{j},T}](\tau_{j},\xi_{j}),
\]
where $c$ is a constant. Therefore, Lemma~\ref{modul_est} implies that
\[
\int_{\R}\int_{\R^{d}}\left(N_{0}\prod_{j=0}^{m}Q_{<M}^{\Delta}u_{j,N_{j},T}\right) dxdt=0.
\]
So, let us now consider the case that $Q_{j}^{\Delta}=Q_{\geq M}^{\Delta}$ for some $0\leq j\leq m$. 

First, we consider the case $Q_{0}^{\Delta}=Q_{\geq M}^{\Delta}$.  
By the Cauchy-Schwartz inequality, we have
\[
\begin{split}
&\left|\sum_{S(N_{1})}\int_{\R}\int_{\R^{d}}\left(N_{0}Q_{\geq M}^{\Delta}u_{0,N_{0},T}\prod_{j=1}^{m}Q_{j}^{\Delta}u_{j,N_{j},T}\right)dxdt\right|\\
&\leq \sum_{S(N_{1})}N_{0}||Q_{\geq M}^{\Delta}u_{0,N_{0},T}||_{L^{2}_{tx}}\left|\left|\prod_{j=1}^{m}Q_{j}^{\Delta}u_{j,N_{j},T}\right|\right|_{L^{2}_{tx}}. 
\end{split}
\]
Furthermore by (\ref{highMproj}) and $M\sim N_{0}^{2}$, we have
\[
||Q_{\geq M}^{\Delta}u_{0,N_{0},T}||_{L^{2}_{tx}}
\lesssim N_{0}^{-1}||u_{0,N_{0},T}||_{V^{2}_{\Delta}}.
\]
While by (\ref{V_Multi}), (\ref{Vproj})
and the Cauchy-Schwartz inequality for the dyadic sum, we have
\[
\begin{split}
\sum_{S(N_{1})}\left|\left|\prod_{j=1}^{m}Q_{j}^{\Delta}u_{j,N_{j},T}\right|\right|_{L^{2}_{tx}}
&\lesssim ||u_{1,N_{1},T}||_{V^{2}_{\Delta}}\sum_{S(N_{1})}\prod_{j=2}^{m}\left(\frac{N_{j}}{N_{1}}\right)^{\delta}N_{j}^{s_{c}}||u_{j,N_{j},T}||_{V^{2}_{\Delta}}\\
&\lesssim ||u_{1,N_{1},T}||_{V^{2}_{\Delta}}\prod_{j=2}^{m}\left(\sum_{N_{j}\leq N_{1}}N_{j}^{2s_{c}}||u_{j,N_{j},T}||_{V^{2}_{\Delta}}^{2}\right)^{1/2}.
\end{split}
\]
Therefore, we obtain
\[
\begin{split}
&\left|\sum_{S(N_{1})}\int_{\R}\int_{\R^{d}}\left(N_{0}Q_{\geq M}^{\Delta}u_{0,N_{0},T}\prod_{j=1}^{m}Q_{j}^{\Delta}u_{j,N_{j},T}\right)dxdt\right|\\
&\lesssim 
||P_{N_{0}}u_{0}||_{V^{2}_{\Delta}}||P_{N_{1}}u_{1}||_{V^{2}_{\Delta}}\prod_{j=2}^{m}||u_{j}||_{\dot{Y}^{s_{c}}}
\end{split}
\]
since $||\ee_{[0,T)}u||_{V^{2}_{\Delta}}\lesssim ||u||_{V^{2}_{\Delta}}$ for any $T\in (0,\infty]$. 
For the case $Q_{1}^{\Delta}=Q_{\geq M}^{\Delta}$ is proved in same way. 

Next, we consider the case $Q_{k}^{\Delta}=Q_{\geq M}^{\Delta}$ for some $2\leq k\leq m$. 
By the H\"older's inequality, we have
\[
\begin{split}
&\left|\sum_{S(N_{1})}\int_{\R}\int_{\R^{d}}\left(N_{0}Q_{\geq M}^{\Delta}u_{k,N_{k},T}\prod_{\substack{j=0\\ j\neq k}}^{m}Q_{j}^{\Delta}u_{j,N_{j},T}\right)dxdt\right|\\
&\lesssim N_{0}||Q_{0}^{\Delta}u_{0,N_{0},T}||_{L_{t}^{4}L_{x}^{2d/(d-1)}}||Q_{1}^{\Delta}u_{1,N_{1},T}||_{L_{t}^{4}L_{x}^{2d/(d-1)}}\\
&\ \ \ \ \times \left|\left| \sum_{N_{k}}Q_{\geq M}^{\Delta}u_{k,N_{k},T}\right|\right|_{L_{t}^{2}L_{x}^{(m-1)d}}
\prod_{\substack{j=2\\ j\neq k}}^{m}\left|\left| \sum_{N_{j}}Q_{j}^{\Delta}u_{j,N_{j},T}\right|\right|_{L_{t}^{\infty}L_{x}^{(m-1)d}}. 
\end{split}
\]
By (\ref{U_Stri}), the embedding $V^{2}_{\Delta}\hookrightarrow U^{4}_{\Delta}$ and (\ref{Vproj}), we have
\[
||Q_{0}^{\Delta}u_{0,N_{0},T}||_{L_{t}^{4}L_{x}^{2d/(d-1)}}||Q_{1}^{\Delta}u_{1,N_{1},T}||_{L_{t}^{4}L_{x}^{2d/(d-1)}}
\lesssim ||u_{0,N_{0},T}||_{V^{2}_{\Delta}}||u_{1,N_{1},T}||_{V^{2}_{\Delta}}.
\]
While by the Sobolev embedding $\dot{H}^{s_{c}}(\R^{d})\hookrightarrow L^{(m-1)d}(\R^{d})$, $L^{2}$ orthogonality and (\ref{highMproj}), we have
\[
\begin{split}
\left|\left| \sum_{N_{k}}Q_{\geq M}^{\Delta}u_{k,N_{k},T}\right|\right|_{L_{t}^{2}L_{x}^{(m-1)d}}
&\lesssim \left(\sum_{N_{k}}N_{k}^{2s_{c}}||Q_{\geq M}^{\Delta}u_{k,N_{k},T}||_{L_{tx}^{2}}^{2}\right)^{1/2}\\
&\lesssim N_{0}^{-1}\left(\sum_{N_{k}}N_{k}^{2s_{c}}||u_{k,N_{k},T}||_{V^{2}_{\Delta}}^{2}\right)^{1/2}
\end{split}
\]
since $M\sim N_{0}^{2}$. Furthermore by the Sobolev embedding $\dot{H}^{s_{c}}(\R^{d})\hookrightarrow L^{(m-1)d}(\R^{d})$, $L^{2}$ orthogonality, 
$V^{2}_{\Delta}\hookrightarrow L^{\infty}(\R;L^{2})$ and (\ref{Vproj}), we have
\[
\begin{split}
\left|\left| \sum_{N_{j}}Q_{j}^{\Delta}u_{j,N_{j},T}\right|\right|_{L_{t}^{\infty}L_{x}^{(m-1)d}}
&\lesssim \left(\sum_{N_{j}}N_{j}^{2s_{c}}||Q_{j}^{\Delta}u_{j,N_{j},T}||_{L_{t}^{\infty}L_{x}^{2}}^{2}\right)^{1/2}\\
&\lesssim \left(\sum_{N_{j}}N_{j}^{2s_{c}}||u_{j,N_{j},T}||_{V^{2}_{\Delta}}^{2}\right)^{1/2}.
\end{split}
\] 
As a result, we obtain
\[
\begin{split}
&\left|\sum_{S(N_{1})}\int_{\R}\int_{\R^{d}}\left(N_{0}Q_{\geq M}^{\Delta}u_{k,N_{k},T}\prod_{\substack{j=0\\ j\neq k}}^{m}Q_{j}^{\Delta}u_{j,N_{j},T}\right)dxdt\right|\\
&\lesssim 
||P_{N_{0}}u_{0}||_{V^{2}_{\Delta}}||P_{N_{1}}u_{1}||_{V^{2}_{\Delta}}\prod_{j=2}^{m}||u_{j}||_{\dot{Y}^{s_{c}}}
\end{split}
\]
since $||\ee_{[0,T)}u||_{V^{2}_{\Delta}}\lesssim ||u||_{V^{2}_{\Delta}}$ for any $T\in (0,\infty]$. 

\end{proof}
\begin{prop}\label{HH_est}
Let $d\geq 1$, $m\geq 3$, $s_{c}=d/2-1/(m-1)$ and $0<T\leq \infty$. 
For a dyadic number $N_{2}\in 2^{\Z}$, we define the set $S_{*}(N_{2})$ as
\[
S_{*}(N_{2}):=\{ (N_{3},\cdots ,N_{m})\in (2^{\Z})^{m-2}|N_{2}\geq N_{3}\geq \cdots \geq N_{m}\}. 
\]
If $N_{0}\lesssim N_{1}\sim N_{2}$, then we have
\begin{equation}\label{hh}
\begin{split}
&\left|\sum_{S_{*}(N_{2})}\int_{0}^{T}\int_{\R^{d}}\left(N_{0}\prod_{j=0}^{m}P_{N_{j}}u_{j}\right)dxdt\right|\\
&\lesssim 
\frac{N_{0}}{N_{1}}||P_{N_{0}}u_{0}||_{V^{2}_{\Delta}}||P_{N_{1}}u_{1}||_{V^{2}_{\Delta}}N_{2}^{s_{c}}||P_{N_{2}}u_{2}||_{V^{2}_{\Delta}}\prod_{j=3}^{m}||u_{j}||_{\dot{Y}^{s_{c}}}. 
\end{split}
\end{equation}
\end{prop}
\begin{proof} 
We define $u_{j,N_{j},T}:=\ee_{[0,T)}P_{N_{j}}u_{j}$\ $(j=1,\cdots ,m)$ and put 
$M:=N_{1}^{2}/4(m+1)$. We decompose $Id=Q^{\Delta}_{<M}+Q^{\Delta}_{\geq M}$. 
We divide the integrals on the left-hand side of (\ref{hh}) into $2^{m+1}$ piece of the form 
\begin{equation}\label{piece_form_hh}
\int_{\R}\int_{\R^{d}}\left(N_{0}\prod_{j=0}^{m}Q_{j}^{\Delta}u_{j,N_{j},T}\right) dxdt
\end{equation}
with $Q_{j}^{\Delta}\in \{Q_{\geq M}^{\Delta}, Q_{<M}^{\Delta}\}$\ $(j=0,\cdots ,m)$. 
By the Plancherel's theorem, we have
\[
(\ref{piece_form_hh})
= c\int_{\sum_{j=0}^{m}\tau_{j}=0}\int_{\sum_{j=0}^{m}\xi_{j}=0}N_{0}\prod_{j=0}^{m}\FT[Q_{j}^{\Delta}u_{j,N_{j},T}](\tau_{j},\xi_{j}),
\]
where $c$ is a constant. Therefore, Lemma~\ref{modul_est} implies that
\[
\int_{\R}\int_{\R^{d}}\left(N_{0}\prod_{j=0}^{m}Q_{<M}^{\Delta}u_{j,N_{j},T}\right) dxdt=0.
\]
So, let us now consider the case that $Q_{j}^{\Delta}=Q_{\geq M}^{\Delta}$ for some $0\leq j\leq m$. 

We consider only for the case $Q_{0}^{\Delta}=Q_{\geq M}^{\Delta}$ since 
the case $Q_{1}^{\Delta}=Q_{\geq M}^{\Delta}$ is similar argument and 
the cases $Q_{k}^{\Delta}=Q_{\geq M}^{\Delta}$ $(k=2,\cdots ,m)$ are similar to the argument in the proof of Proposition~\ref{HL_est}.  
By the H\"older's inequality and we have
\[
\begin{split}
&\left|\sum_{S_{*}(N_{2})}\int_{\R}\int_{\R^{d}}\left(N_{0}Q_{\geq M}^{\Delta}u_{0,N_{0},T}\prod_{\substack{j=1}}^{m}Q_{j}^{\Delta}u_{j,N_{j},T}\right)dxdt\right|\\
&\lesssim N_{0}||Q_{\geq M}^{\Delta}u_{0,N_{0},T}||_{L_{t}^{2}L_{x}^{(m-1)d}}||Q_{1}^{\Delta}u_{1,N_{1},T}||_{L_{t}^{4}L_{x}^{2d/(d-1)}}||Q_{2}^{\Delta}u_{2,N_{2},T}||_{L_{t}^{4}L_{x}^{2d/(d-1)}}\\
&\ \ \ \ \times \prod_{j=3}^{m}\left|\left| \sum_{N_{j}}Q_{j}^{\Delta}u_{j,N_{j},T}\right|\right|_{L_{t}^{\infty}L_{x}^{(m-1)d}}. 
\end{split}
\]
By the Sobolev embedding $\dot{H}^{s_{c}}(\R^{d})\hookrightarrow L^{(m-1)d}(\R^{d})$ and (\ref{highMproj}), we have
\[
\begin{split}
\left|\left|Q_{\geq M}^{\Delta}u_{0,N_{0},T}\right|\right|_{L_{t}^{2}L_{x}^{(m-1)d}}
&\lesssim N_{0}^{s_{c}}||Q_{\geq M}^{\Delta}u_{0,N_{0},T}||_{L_{tx}^{2}}\\
&\lesssim N_{1}^{-1}N_{2}^{s_{c}}||u_{0,N_{0},T}||_{V^{2}_{\Delta}}
\end{split}
\]
since $M\sim N_{1}^{2}$ and $N_{0}\lesssim N_{2}$. 
While by (\ref{U_Stri}), the embedding $V^{2}_{\Delta}\hookrightarrow U^{4}_{\Delta}$ and (\ref{Vproj}), we have
\[
||Q_{1}^{\Delta}u_{1,N_{1},T}||_{L_{t}^{4}L_{x}^{2d/(d-1)}}||Q_{2}^{\Delta}u_{2,N_{2},T}||_{L_{t}^{4}L_{x}^{2d/(d-1)}}
\lesssim ||u_{1,N_{1},T}||_{V^{2}_{\Delta}}||u_{2,N_{2},T}||_{V^{2}_{\Delta}}.
\]
Furthermore by the Sobolev embedding $\dot{H}^{s_{c}}(\R^{d})\hookrightarrow L^{(m-1)d}(\R^{d})$, $L^{2}$ orthogonality, 
$V^{2}_{\Delta}\hookrightarrow L^{\infty}(\R;L^{2})$ and (\ref{Vproj}), we have
\[
\begin{split}
\left|\left| \sum_{N_{j}}Q_{j}^{\Delta}u_{j,N_{j},T}\right|\right|_{L_{t}^{\infty}L_{x}^{(m-1)d}}
&\lesssim \left(\sum_{N_{j}}N_{j}^{2s_{c}}||Q_{j}^{\Delta}u_{j,N_{j},T}||_{L_{t}^{\infty}L_{x}^{2}}^{2}\right)^{1/2}\\
&\lesssim \left(\sum_{N_{j}}N_{j}^{2s_{c}}||u_{j,N_{j},T}||_{V^{2}_{\Delta}}^{2}\right)^{1/2}.
\end{split}
\] 
As a result, we obtain
\[
\begin{split}
&\left|\sum_{S_{*}(N_{2})}\int_{\R}\int_{\R^{d}}\left(N_{0}Q_{\geq M}^{\Delta}u_{0,N_{0},T}\prod_{j=1}^{m}Q_{j}^{\Delta}u_{j,N_{j},T}\right)dxdt\right|\\
&\lesssim 
\frac{N_{0}}{N_{1}}||P_{N_{0}}u_{0}||_{V^{2}_{\Delta}}||P_{N_{1}}u_{1}||_{V^{2}_{\Delta}}N_{2}^{s_{c}}||P_{N_{2}}u_{2}||_{V^{2}_{\Delta}}\prod_{j=2}^{m}||u_{j}||_{\dot{Y}^{s_{c}}}
\end{split}
\]
since $||\ee_{[0,T)}u||_{V^{2}_{\Delta}}\lesssim ||u||_{V^{2}_{\Delta}}$ for any $T\in (0,\infty]$. 
%
\end{proof}
%
%
\section{Proof of the well-posedness and the scattering \label{proof_thm}}\kuuhaku
In this section, we prove Theorem~\ref{wellposed_1} and Corollary~\ref{sccat}. 
We define the map $\Phi_{T, \varphi}$ as
\[
\Phi_{T, \varphi}(u)(t):=e^{it\Delta}\varphi -iI_{T}(u,\cdots, u)(t),
\] 
where
\[
I_{T}(u_{1},\cdots u_{m})(t):=\int_{0}^{t}\ee_{[0,T)}(t')e^{i(t-t')\Delta}\partial_{k}\left(\prod_{j=1}^{m}\overline{u_{j}(t')}\right)dt'.
\]
To prove the well-posedness of (\ref{mDNLS}), we prove that $\Phi_{T, \varphi}$ is a contraction map 
on a closed subset of $\dot{Z}^{s}([0,T])$ or $Z^{s}([0,T])$. 
Key estimate is the following:
\begin{prop}\label{Duam_est}
We assume $d\geq 1$, $m\geq 3$. Then for $s_{c}=d/2-1/(m-1)$ and any $0<T\leq \infty$, we have
\begin{equation}\label{Duam_est_1}
||I_{T}(u_{1},\cdots u_{m})||_{\dot{Z}^{s_{c}}}\lesssim \prod_{j=1}^{m}||u_{j}||_{\dot{Y}^{s_{c}}}.
\end{equation}
\end{prop}
\begin{proof}
We show the estimate
\begin{equation}\label{Duam_est_2}
||I_{T}(u_{1},\cdots u_{m})||_{\dot{Z}^{s}}\lesssim \sum_{k=1}^{m}\left(||u_{k}||_{\dot{Y}^{s}}\prod_{\substack{j=1\\ j\neq k}}^{m}||u_{j}||_{\dot{Y}^{s_{c}}}\right)
\end{equation}
for $s\geq 0$. (\ref{Duam_est_1}) follows from (\ref{Duam_est_2}) with $s=s_{c}$. 
We decompose
\[
I_{T}(u_{1},\cdots u_{m})=\sum_{N_{1},\cdots ,N_{m}}I_{T}(P_{N_{1}}u_{1},\cdots P_{N_{m}}u_{m}).
\]
By symmetry, it is enough to consider the summation for $N_{1}\geq \cdots \geq N_{m}$. We put
\[
\begin{split}
S_{1}&:=\{ (N_{1},\cdots ,N_{m})\in (2^{\Z})^{m}|N_{1}\gg N_{2}\geq \cdots \geq N_{m}\}\\
S_{2}&:=\{ (N_{1},\cdots ,N_{m})\in (2^{\Z})^{m}|N_{1}\sim N_{2}\geq \cdots \geq N_{m}\}
\end{split}
\]
and
\[
J_{k}:=\left|\left| \sum_{S_{k}}I_{T}(P_{N_{1}}u_{1},\cdots P_{N_{m}}u_{m})\right|\right|_{\dot{Z}^{s}}\ (k=1,2).
\]

First, we prove the estimate for $J_{1}$. By Theorem~\ref{duality} and the Plancherel's theorem, we have
\[
\begin{split}
J_{1}&\leq \left\{ \sum_{N_{0}}N_{0}^{2s}\left|\left| e^{-it\Delta}P_{N_{0}}\sum_{S_{1}}I_{T}(P_{N_{1}}u_{1},\cdots P_{N_{m}}u_{m})\right|\right|_{U^{2}}^{2}\right\}^{1/2}\\
&\lesssim \left\{\sum_{N_{0}}N_{0}^{2s}\sum_{N_{1}\sim N_{0}}
\left( \sup_{||u_{0}||_{V^{2}_{\Delta}}=1}\left|\sum_{S(N_{1})}\int_{0}^{T}\int_{\R^{d}}\left(N_{0}\prod_{j=0}^{m}P_{N_{j}}u_{j}\right)dxdt\right|\right)^{2}\right\}^{1/2}. 
\end{split}
\]
Therefore by Proposition~\ref{HL_est}, we have
\[
\begin{split}
J_{1}&\lesssim \left\{\sum_{N_{0}}N_{0}^{2s}\sum_{N_{1}\sim N_{0}}
\left( \sup_{||u_{0}||_{V^{2}_{\Delta}}=1}||P_{N_{0}}u_{0}||_{V^{2}_{\Delta}}||P_{N_{1}}u_{1}||_{V^{2}_{\Delta}}\prod_{j=2}^{m}||u_{j}||_{\dot{Y}^{s_{c}}}\right)^{2}\right\}^{1/2}\\
&\lesssim 
\left(\sum_{N_{1}}N_{1}^{2s}||P_{N_{1}}u_{1}||_{V^{2}_{\Delta}}^{2}\right)^{1/2}
\prod_{j=2}^{m}||u_{j}||_{\dot{Y}^{s_{c}}}\\
&=||u_{1}||_{\dot{Y}^{s}}\prod_{j=2}^{m}||u_{j}||_{\dot{Y}^{s_{c}}}.
\end{split}
\]

Next, we prove the estimate for $J_{2}$. By Theorem~\ref{duality} and the Plancherel's theorem, we have
\[
\begin{split}
J_{2}&\leq
\sum_{N_{1}}\sum_{N_{2}\sim N_{1}}\left(\sum_{N_{0}}N_{0}^{2s}\left|\left|e^{-it\Delta}P_{N_{0}}\sum_{S_{*}(N_{2})}I_{T}(P_{N_{1}}u_{1},\cdots P_{N_{m}}u_{m})\right|\right|_{U^{2}}^{2}\right)^{1/2}\\
&=\sum_{N_{1}}\sum_{N_{2}\sim N_{1}}\left(\sum_{N_{0}\lesssim N_{1}}N_{0}^{2s}
\sup_{||u_{0}||_{V^{2}_{\Delta}}=1}\left| \sum_{S_{*}(N_{2})}\int_{0}^{T}\int_{\R^{d}}\left(N_{0}\prod_{j=0}^{m}P_{N_{j}}u_{j}\right)dxdt\right|^{2}\right)^{1/2}.
\end{split}
\]
Therefore by {\rm Proposition~\ref{HH_est}} and Cauchy-Schwartz inequality for dyadic sum, we have
\[
\begin{split}
J_{2}&\lesssim
\sum_{N_{1}}\sum_{N_{2}\sim N_{1}}\left(\sum_{N_{0}\lesssim N_{1}}N_{0}^{2s}
\left(\frac{N_{0}}{N_{1}}||P_{N_{1}}u_{1}||_{V^{2}_{\Delta}}N_{2}^{s_{c}}||P_{N_{2}}u_{2}||_{V^{2}_{\Delta}}\prod_{j=3}^{m}||u_{j}||_{\dot{Y}^{s_{c}}}\right)^{2}\right)^{1/2}\\
&\lesssim \left(\sum_{N_{1}}N_{1}^{2s}||P_{N_{1}}u_{1}||_{V^{2}_{\Delta}}^{2}\right)^{1/2}
\left(\sum_{N_{2}}N_{2}^{2s_{c}}||P_{N_{2}}u_{2}||_{V^{2}_{\Delta}}^{2}\right)^{1/2}\prod_{j=3}^{m}||u_{j}||_{\dot{Y}^{s_{c}}}\\
&= ||u_{1}||_{\dot{Y}^{s}}\prod_{j=2}^{m}||u_{j}||_{\dot{Y}^{s_{c}}}.
\end{split}
\]
\end{proof}
The estimates (\ref{Duam_est_2}) with $s=0$ and with $s=s_{c}$ imply the following.  
\begin{cor}\label{Duam_est_inhom}
We assume $d\geq 1$, $m\geq 3$. Then for $s\geq s_{c}$ $(=d/2-1/(m-1))$ and any $0<T\leq \infty$, we have
\[
||I_{T}(u_{1},\cdots u_{m})||_{Z^{s}}\lesssim \prod_{j=1}^{m}||u_{j}||_{Y^{s}}.
\]
\end{cor}
\begin{proof}[\rm{\bf{Proof of Theorem~\ref{wellposed_1}.}}]
We prove only the homogeneous case. The inhomogeneous case is also proved by the same way. 
For $r>0$, we define 
\begin{equation}\label{Zr_norm}
\dot{Z}^{s}_{r}(I)
:=\left\{u\in \dot{Z}^{s}(I)\left|\ ||u||_{\dot{Z}^{s}(I)}\leq 2r \right.\right\}
\end{equation}
which is a closed subset of $\dot{Z}^{s}(I)$. 
Let $u_{0}\in B_{r}(\dot{H}^{s_{c}})$ be given. For $u\in \dot{Z}^{s_{c}}_{r}([0,\infty ))$, 
we have
\[
||\Phi_{T,u_{0}}(u)||_{\dot{Z}^{s_{c}}([0,\infty ))}\leq ||u_{0}||_{\dot{H}^{s_{c}}} +C||u||_{\dot{Z}^{s_{c}}([0,\infty ))}^{m}\leq r(1+2^{m}Cr^{m-1})
\]
and
\[
\begin{split}
||\Phi_{T,u_{0}}(u)-\Phi_{T,u_{0}}(v)||_{\dot{Z}^{s_{c}}([0,\infty ))}
&\leq C(||u||_{\dot{Z}^{s_{c}}([0,\infty ))}+||v||_{\dot{Z}^{s_{c}}([0,\infty ))})^{m-1}||u-v||_{\dot{Z}^{s_{c}}([0,\infty ))}\\
&\leq 4^{m-1}Cr^{m-1}||u-v||_{\dot{Z}^{s_{c}}([0,\infty ))}
\end{split}
\]
by Proposition~\ref{Duam_est} and
\[
||e^{it\Delta}\varphi ||_{\dot{Z}^{s_{c}}([0,\infty ))}\leq ||\ee_{[0,\infty )}e^{it\Delta}\varphi ||_{\dot{Z}^{s_{c}}_{\sigma}}\leq ||\varphi ||_{\dot{H}^{s_{c}}}, 
\] 
where $C$ is an implicit constant in (\ref{Duam_est_1}). Therefore if we choose $r$ satisfying
\[
r <(4^{m-1}C)^{-1/(m-1)},
\]
then $\Phi_{T,u_{0}}$ is a contraction map on $\dot{Z}^{s_{c}}_{r}([0,\infty ))$. 
This implies the existence of the solution of (\ref{mDNLS}) and the uniqueness in the ball $\dot{Z}^{s_{c}}_{r}([0,\infty ))$. 
The Lipschitz continuously of the flow map is also proved by similar argument. 
\end{proof} 
\begin{proof}[\rm{\bf{Proof of Corollary~\ref{sccat}.}}]
We prove only the homogeneous case. The inhomogeneous case is also proved by the same way. 
By Proposition~\ref{Duam_est}, 
the global solution $u\in \dot{Z}^{s_{c}}([0,\infty ))$ of (\ref{mDNLS}) which was constructed in Theorem~\ref{wellposed_1} 
satisfies
\[
N^{s_{c}}e^{-it\Delta}P_{N}I_{\infty }(u,\cdots ,u)\in V^{2}
\]
for each $N\in 2^{\Z}$. This implies that
\[
u_{+}:=\lim_{t\rightarrow \infty}(u_{0}-e^{-it\Delta}I_{\infty}(u,\cdots ,u)(t))
\]
exists in $\dot{H}^{s_{c}}$ by Proposition~\ref{upvpprop}.\ {\rm (4)}. 
Then we obtain
\[
u-e^{it\Delta}u_{+}\rightarrow 0
\]
in $\dot{H}^{s_{c}}$ as $t\rightarrow \infty$. 
\end{proof}
\section*{acknowledgements}
The author would like to express his appreciation to Kotaro Tsugawa for many discussions and very valuable comments. 


\begin{thebibliography}{10}
\bibitem{BL01}H. Biagioni and F. Linares, {\itshape Ill-posedness for the derivative Schr\"odinger and generalized Benjamin-Ono equations},
Trans. Amer. Math. Soc., {\bfseries 353}(2001), no.9, 3649--3659.
\bibitem{Be06}I. Bejenaru, {\itshape Quadratic nonlinear derivative Schr\"odinger equations. Part I},
Int. Math. Res. Pap., (2006), Art. ID 70630, 5925--5957.
\bibitem{Be08}I. Bejenaru, {\itshape Quadratic nonlinear derivative Schr\"odinger equations. Part II},
Trans. Amer. Math. Soc., {\bfseries 360}(2008), no.11, 84pp.
\bibitem{Chi95}H. Chihara, {\itshape Local existence for semilinear Schr\"odinger equations}, Math. Japon., {\bfseries 42} (1995), 35--51. 
\bibitem{Chi99}H. Chihara, {\itshape Gain of regularity for semilinear Schr\"odinger equations}, Math. Ann., {\bfseries 315} (1999), 529--567. 
\bibitem{Ch}M. Christ, {\itshape Illposedness of a Schr\"odinger equation with derivative nonlinearity}, preprint  (http://citeseerx.ist.psu.edu/viewdoc/summary?doi=10.1.1.70.1363). 
\bibitem{CDKS01}J. Colliander, J. Delort, C. Kenig and G. Staffilani, {\itshape Bilinear estimates and applications to 2D NLS},
  Trans. Amer. Math. Soc., {\bfseries 353} (2001), no.8, 3307--3325. 
\bibitem{CKSTT01}J. Colliander, M. Keel, G. Staffilani, H. Takaoka and T. Tao, {\itshape Global well-posedness result
for Schr\"odigner equations with derivative},
  SIAM J. Math. Anal., {\bfseries 33} (2001), no.2, 649--669. 
\bibitem{CKSTT02}J. Colliander, M. Keel, G. Staffilani, H. Takaoka and T. Tao, {\itshape A refined global well-posedness result
for Schr\"odigner equations with derivative},
  SIAM J. Math. Anal., {\bfseries 34} (2002), no.1, 64--86. 
\bibitem{Gr}A. Gr\"unrock, {\itshape On the Cauchy - and periodic boundary value problem for a certain class of derivative nonlinear Schr\"odinger equations}, 
preprint (arXiv:0006195v1 [math.AP]). 
\bibitem{HHK09}M. Hadac, S. Herr and H. Koch, {\itshape Well-posedness and scattering for the KP-II equation in a critical space},
  Ann. Inst. H. Poincar\'e Anal. Non lin\'eaie., {\bfseries 26} (2009), no.3, 917--941. 
\bibitem{HHK10}M. Hadac, S. Herr and H. Koch, {\itshape Errantum to ``Well-posedness and scattering for the KP-II equation in a critical space'' [Ann. I. H. Poincar\'e--AN26 (3) (2009) 917--941]}, Ann. Inst. H. Poincar\'e Anal. Non lin\'eaie., {\bfseries 27} (2010), no.3, 971--972. 
\bibitem{Ha07}C. Hao, {\itshape Well-posedness for one-dimensional derivative nonlinear Schr\"odinger equations},
  Comm. Pure Appl. Anal., {\bfseries 6} (2007), no.4, 997--1021. 
\bibitem{Hi}H. Hirayama, {\itshape Well-posedness  and scattering for a system of quadratic derivative
nonlinear Schr\"odinger equations with low regularity initial data}, preprint 
(arXiv:1309.4336 [math.AP]). 
\bibitem{IK07}A. Ionescu and C. Kenig, {\itshape Global well-posedness of the Benjamin-Ono equation in low-regularity spaces}, J. Amer. Math. Soc., {\bfseries 20} (2007), no. 3, 753--798. 
\bibitem{KPV98}C. Kenig, G. Ponce and L. Vega, {\itshape Smoothing effects and local existence theory for the generalized nonlinear Schr\"odinger equations},
  Invent. Math., {\bfseries 134} (1998), no.3, 489--545. 
\bibitem{KT05}H. Koch and N. Tzvetkov, {\itshape Nonlinear wave interactions for the Benjamin-Ono equation}, Int. Math. Res. Not., {\bfseries 2005} (2005), no.30, 1833--1847. 
\bibitem{MWX11}C. Miao, Y. Wu, G. Xu, {\itshape Global well-posedness for Schrodinger equation with derivative in $H^{1/2}(\R )$}, J. Diff. Eqns., {\bfseries 251} (2011), no.8, 2164--2195. 
\bibitem{Mi85}S. Mizohata, {\itshape On the Cauchy problem}, Notes and Reports in Mathematics in Science and Engineering, Science Press \& Academic Press., {\bfseries 3} (1985), no.3, 177. 
\bibitem{Oz98}T. Ozawa, {\itshape Finite energy solutions for the Schr\"odinger equations with
quadratic nonlinearity in one space dimension}, Funkcialaj Ekvacioj., {\bfseries 41} (1998), 451--468.
\bibitem{St07}A. Stefanov, {\itshape On quadratic derivative Schr\"odinger equations in one space dimension}, Trans. Amer. Math. Soc., {\bfseries 359} (2007), no. 8, 3589--3607. 
\bibitem{Tak99}H. Takaoka, {\itshape Well-posedness for the one dimensional Schr\"odinger equation with the derivative nonlinearity}, Adv. Diff. Eqns., {\bfseries 4} (1999), 561--680. 
\bibitem{Tak01}H. Takaoka, {\itshape Global well-posedness for Schr\"odinger equations with derivative in a nonlinear term and data in low-order Sobolev spaces}, Electron. J. Diff. Eqns., {\bfseries 42} (2001), 1--23. 
\bibitem{Ta04}T. Tao, {\itshape Global well-posedness of the Benjamin--Ono equation in $H^{1}(\R )$}, J. Hyperbolic Differ. Equ., {\bfseries 1} (2004), no. 3, 27--49. 
\end{thebibliography}
\end{document}